\documentclass{article}
\usepackage{amssymb,amsmath,latexsym,amsthm}
\usepackage[english]{babel}
\usepackage[mathscr]{eucal}

\allowdisplaybreaks

\newtheorem{theorem}{Theorem}
\renewcommand{\thetheorem}{\Alph{theorem}}

\newtheorem{lemma}{Lemma}

\newtheorem*{maintheorem}{MAIN THEOREM}
\newtheorem*{proposition}{Proposition}
\newtheorem*{TPC}{Twin Prime Conjecture}
\newtheorem*{PC}{Polignac's Conjecture}
\newtheorem*{ConjDHLk2}{Conjecture DHL {\boldmath$(k,2)$}}
\newtheorem*{ConjDHLstark2}{Conjecture DHL* {\boldmath$(k,2)$}}

\theoremstyle{definition}
\newtheorem{definition}{Definition}

\theoremstyle{remark}

\newtheorem*{remark}{Remark}

\def\beq{\begin{equation}}
\def\eeq{\end{equation}}

\numberwithin{equation}{section}
\begin{document}

\title{Polignac Numbers, Conjectures of Erd\H{o}s on Gaps between Primes, Arithmetic Progressions in Primes, and the Bounded Gap Conjecture}

\author{by\\
J\'anos PINTZ\\
R\'enyi Mathematical Institute of the Hungarian Academy
of Sciences\\
Budapest, Re\'altanoda u. 13--15\\
H-1053 Hungary\\
e-mail: {pintz@renyi.hu}}

\date{}

\footnotetext{Supported by OTKA grants K~100291, NK~104183 and ERC-AdG.\ 228005.}

\numberwithin{equation}{section}

\maketitle

\begin{abstract}
In the present work we prove a number of surprising results about gaps
between consecutive primes and arithmetic progressions in the sequence of
generalized twin primes which could not have been proven without the recent
fantastic achievement of Yitang Zhang about the existence of bounded gaps
between consecutive primes.
Most of these results would have belonged to the category of science fiction a decade ago.
However, the presented results are far from being immediate consequences of Zhang's famous theorem: they require
various new ideas, other important properties of the applied sieve function and
a closer analysis of the methods of Goldston--Pintz--Y{\i}ld{\i}r{\i}m, Green--Tao, and
Zhang, respectively.

\end{abstract}

\section{Introduction}
\label{sec:1}

1.1. Very recently Yitang Zhang, in a fantastic breakthrough solved the Bounded Gap Conjecture, a term formulated in recent works of mine, often in collaboration with D. Goldston and C. Y{\i}ld{\i}r{\i}m.
Let $p_n$ denote the $n$\textsuperscript{th} prime, $\mathcal P$ the set of all primes and
\beq
\label{eq:1.1}
d_n = p_{n + 1} - p_n
\eeq
the $n$\textsuperscript{th} difference between consecutive primes.

\begin{theorem}[{Y. Zhang \cite[2013]{Zhang2013}}]
\label{th:A}
$\liminf\limits_{n \to \infty} d_n \leq 7\cdot 10^7$.
\end{theorem}

Earlier, in a joint work with D. Goldston and C. Y{\i}ld{\i}r{\i}m we showed \cite[2009]{GPY2009} this under the deep unproved condition that primes have a distribution level greater than $1/2$, in other words, under the assumption that the exponent $1/2$ in the Bombieri--Vinogradov Theorem can be improved.

\begin{theorem}[{\cite[2009]{GPY2009}}]
\label{th:B}
If the primes have a distribution level $\vartheta > 1/2$, then with a suitable explicit constant $C(\vartheta)$ we have
\beq
\label{eq:1.2}
\liminf_{n \to \infty} d_n \leq C(\vartheta).
\eeq
If the Elliott--Halberstam Conjecture (EH) is true, i.e. $\vartheta = 1$ or, at least, $\vartheta > 0.971$ then
\beq
\label{eq:1.3}
\liminf_{n \to \infty} d_n \leq 16.
\eeq
\end{theorem}

We say that $\vartheta$ is a level of distribution of primes if
\beq
\label{eq:1.4}
\sum_{q \leq X^{\vartheta - \varepsilon}} \max_{\substack{a\\ (a, q) = 1}} \Biggl| \sum_{\substack{p \equiv a(\text{\rm mod }q)\\
p\leq X}} \log p - \frac{X}{\varphi(q)}\Biggr| \ll_{A, \varepsilon} \frac{X}{(\log X)}
\eeq
for any $A, \varepsilon > 0$.
The Bombieri--Vinogradov Theorem asserts that $\vartheta = 1/2$ is a level of distribution and the Elliott--Halberstam Conjecture
\cite[1970]{EH1970} asserts that $\vartheta = 1$ is also an admissible level.
Zhang could not prove \eqref{eq:1.4} with a level $> 1/2$ but showed that

(i) it is possible to neglect in \eqref{eq:1.4} all moduli~$q$ with a prime divisor $> X^b$, where $b$ is any constant, and

(ii) managed to show for these smooth moduli a result which is similar but weaker than the analogue of \eqref{eq:1.4} with $\vartheta = 1/2 + 1/584$, which, however suits to apply our method of Theorem~\ref{th:B}.

I have to remark that the present author mentioned the phenomenon (i) in his lecture at the conference of the American Institute of Mathematics in November 2005, Palo Alto, and worked it out in a joint paper with Y. Motohashi \cite[2008]{MP2008}, however, without being able to show~(ii).

The present author worked out recently various consequences of a hypothetical distribution level $\vartheta > 1/2$ (\cite[2010]{Pin2010a} and
\cite[2013]{Erdoskotet}) which partly appeared with proofs \cite[2010]{Pin2010a}, partly exist at a level of announcements \cite[2013]{Erdoskotet}.
The purpose of the present work is to show that most of these conditional results can be shown unconditionally using Zhang's fantastic
result \cite[2013]{Zhang2013} or in some sense, his method, coupled with other ideas.

\section{History of the problems. Formulation of the results}
\label{sec:2}

\subsection{Approximations to the Twin Prime Conjecture}
\label{sec:2.1}

The problem of finding small gaps between primes originates from the

\begin{TPC}
$\liminf\limits_{n \to \infty} d_n = 2$.
\end{TPC}

Since by the Prime Number Theorem the average gap size is $\log n$, Hardy and Littlewood considered first already in 1926 in an unpublished manuscript (see, however \cite[1940]{Ran1940}) the upper estimation of
\beq
\label{eq:2.1}
\Delta_1 = \liminf_{n \to \infty} \frac{d_n}{\log n}
\eeq
and showed $\Delta_1 \leq 2/3$ under the assumption of the Generalized Riemann Hypothesis (GRH).

Erd\H{o}s \cite[1940]{Erd1940} was the first to show unconditionally
\beq
\label{eq:2.2}
\Delta_1 < 1 - c_0, \quad c_0 > 0
\eeq
with an unspecified but explicitly calculable positive constant~$c_0$.
The full history with about 12--15 improvements concerning the value of $\Delta_1$ is contained in \cite[2009]{GPY2009},
so we list here just the most important steps:
\begin{align}
\label{eq:2.3}
\Delta_1 &< 0.4666 \ \text{ (Bombieri, Davenport \cite[1966]{BD1966})},\\
\label{eq:2.4}
\Delta_1 &< 0.2485 \ \text{ (H. Maier \cite[1988]{Mai1988})}.
\end{align}

This was the best result until 2005 when we showed in a joint work with D. Goldston and C. Y{\i}ld{\i}r{\i}m what we called the Small Gap Conjecture:

\begin{theorem}[{\cite[2009]{GPY2009}}]
\label{th:C}
$\Delta_1 = 0$.
\end{theorem}

Soon after it we improved this to

\begin{theorem}[{\cite[2010]{GPY2010}}]
\label{th:D}
We have
\beq
\label{eq:2.5}
\liminf_{n \to \infty} \frac{d_n}{(\log n)^{1/2}(\log\log n)^2} < \infty.
\eeq
\end{theorem}

Finally a few years ago I improved the exponent to $3/7$ and announced the result at the Journ\'ees Arithm\'etiques, Vilnius, 2011, and Tur\'an Memorial Conference, Budapest, 2011.

\begin{theorem}[{\cite[2013]{Pin2013}}]
\label{th:E}
\beq
\label{eq:2.6}
\liminf_{n \to \infty} \frac{d_n}{(\log n)^{3/7}(\log\log n)^{4/7}} < \infty.
\eeq
\end{theorem}

In a joint work with B. Farkas and Sz. Gy. R\'ev\'esz \cite[2013]{FPR2013}
we also showed that essential new ideas are necessary to improve \eqref{eq:2.6}.

\subsection{Polignac numbers}
\label{sec:2.2}

The Twin Prime Conjecture appeared already in a more general form in 1849 in a work of de Polignac.

First we give two definitions.

\begin{definition}
\label{def:1}
A positive even number $2k$ is a strong Polignac number, or briefly a Polignac number if $d_n = 2k$ for infinitely many values of~$n$.
\end{definition}

\begin{definition}
\label{def:2}
A positive even number $2k$ is a weak Polignac number if it can be written as the difference of two primes in an infinitude of ways.
\end{definition}

The set of (strong) Polignac numbers will be denoted by $\mathcal D_s$, the set of weak Polignac numbers by $\mathcal D_w$.
(We have trivially $\mathcal D_s \subseteqq \mathcal D_w$.)

\begin{PC}[{\cite[1849]{Pol1849}}]
Every positive even integer is a (strong) Polignac number.
\end{PC}

Since the smallest weak Polignac number has to be a strong Polignac number, an easy consideration gives that using $|A|$ (or sometimes $\#A$) to denote the number of elements of a set $A$, the following proposition is true.

\begin{proposition}
The following three statements are equivalent:
\begin{itemize}
\item[{\rm (i)}]
the Bounded Gap Conjecture is true;
\item[{\rm (ii)}]
there is at least one (strong) Polignac number, i.e.\ $|\mathcal D_s| \geq 1$;
\item[{\rm (iii)}]
there is at least one weak Polignac number, i.e.\ $|\mathcal D_w| \geq 1$.
\end{itemize}
\end{proposition}

The above very simple proposition shows that the Bounded Gap Conjecture itself leaves still many problems open about Polignac numbers or weak Polignac numbers, but solves the crucial problem that their set is nonempty.
We will prove several unconditional results about the density and distribution of (strong) Polignac numbers.

\renewcommand{\thetheorem}{\arabic{theorem}}
\setcounter{theorem}{0}
\begin{theorem}
\label{th:1}
There exists an explicitly calculable constant $c$ such that for $N > N_0$ we have at least $cN$ Polignac numbers below $N$, i.e.\ Polignac numbers have a positive lower asymptotic density.
\end{theorem}

\begin{theorem}
\label{th:2}
There exists an ineffective constant $C'$ such that every interval of type $[M, M + C']$ contains at least one Polignac number.
\end{theorem}

\begin{remark}
As mentioned earlier, the term Polignac number means always strong Polignac numbers.
\end{remark}

\subsection{The normalized value distribution of $d_n$}
\label{sec:2.3}

The Prime Number Theorem implies
\beq
\label{eq:2.7}
\lim_{N \to \infty} \frac1N \sum \frac{d_n}{\log n} = 1,
\eeq
so it is natural to investigate the series $d_n/\log n$.
Denoting by $J$ the set of limit points of $d_n / \log n$, Erd\H{o}s \cite[1955]{Erd1955} conjectured
\beq
\label{eq:2.8}
J = \left\{ \frac{d_n}{\log n} \right\}' = [0, \infty].
\eeq

While Westzynthius \cite[1931]{Wes1931} proved more than 80 years ago that $\infty \in J$, no finite limit point was known until 2005 when we showed the Small Gap Conjecture, i.e.\ Theorem~\ref{th:C} \cite[2009]{GPY2009}, which is equivalent to $0 \in J$.

Interestingly enough Erd\H{o}s \cite[1955]{Erd1955} and Ricci \cite[1954]{Ric1954} proved independently about 60 years ago that $J$ has a positive Lebesgue measure.
What I can show is a weaker form of Erd\H{o}s's conjecture \eqref{eq:2.8}.

\begin{theorem}
\label{th:3}
There is an ineffective constant $c > 0$ such that
\beq
\label{eq:2.9}
[0, c] \subset J.
\eeq
\end{theorem}

K\'alm\'an Gy\H{o}ry asked me at the Tur\'an Memorial Conference in Budapest, 2011 whether it is possible to find a form of the above result which gives answers about the more subtle value-distribution of $d_n$ if we use a test-function $f(n) \leq \log n$, $f(n) \to \infty$ as $n \to \infty$.

\begin{definition}
\label{def:3}
Let $\mathcal F$ denote the class of functions $f: \ \mathbb Z^+ \to R^+$ with a slow oscillation, when for every $\varepsilon > 0$ we have an $N(\varepsilon) > 0$ such that
\beq
\label{eq:2.10}
(1 - \varepsilon) f(N) \leq f(n) \leq (1 + \varepsilon) f(N) \ \text{ for } \ N \leq n \leq 2N, \quad N > N(\varepsilon).
\eeq
\end{definition}

\begin{theorem}
\label{th:4}
For every function $f \in \mathcal F$, $f(n) \leq \log n$, $\lim\limits_{n \to \infty} f(n) = \infty$ we have an ineffective constant $c_f > 0$ such that
\beq
\label{eq:2.11}
[0, c_f] \subset J_f := \left\{ \frac{d_n}{f(n)}\right\}'.
\eeq
\end{theorem}

\subsection{Comparison of two consecutive values of $d_n$}
\label{sec:2.4}

Erd\H{o}s and Tur\'an proved 65 years ago \cite[1948]{ET1948} that $d_{n + 1} - d_n$ changes sign infinitely often.
This was soon improved by Erd\H{o}s \cite[1948]{Erd1948} to
\beq
\label{eq:2.12}
\liminf_{n \to \infty} \frac{d_{n + 1}}{d_n} < 1 < \limsup_{n \to \infty} \frac{d_{n + 1}}{d_n}.
\eeq
Erd\H{o}s \cite[1955]{Erd1955} wrote seven years later: ``One would of course conjecture that
\beq
\label{eq:2.13}
\liminf_{n \to \infty} \frac{d_{n + 1}}{d_n} = 0, \quad \limsup_{n \to \infty} \frac{d_{n + 1}}{d_n} = \infty,
\eeq
but these conjectures seem very difficult to prove.''

In Section 6 I will show this in a much stronger form:

\begin{theorem}
\label{th:5}
We have
\beq
\label{eq:2.14}
\liminf_{n \to \infty} \frac{d_{n + 1} / d_n}{(\log n)^{-1}} < \infty
\eeq
and
\beq
\label{eq:2.15}
\limsup_{n \to \infty} \frac{d_{n + 1}/d_n}{\log n} > 0.
\eeq
\end{theorem}

\subsection{Arithmetic progressions in the sequence of generalized twin primes}
\label{sec:2.5}

Based on the method of I. M. Vinogradov \cite[1937]{Vin1937} van der Corput \cite[1939]{VdC1939} showed the existence of infinitely many 3-term arithmetic progressions in the sequence of primes.
The problem of the existence of infinitely many $k$-term arithmetic progressions was open for all $k \geq 4$ until 2004, when B. Green and T. Tao \cite[2008]{GT2008} found their wonderful result that primes contain $k$-term arithmetic progressions for every~$k$.
I found recently \cite[2010]{Pin2010a}
a common generalization of Green--Tao's result and our Theorem B \cite[2010]{GPY2010} under the deep assumption that there is a level $\vartheta > 1/2$ of the distribution of primes.
Combining Zhang's method with that of \cite[2010]{Pin2010a} I can prove now the following

\begin{theorem}
\label{th:6}
There is a $d \leq 7 \times 10^7$ such that there are arbitrarily long arithmetic progressions of primes with the property that $p' = p + d$ is the prime following $p$ for each element of the progression.
\end{theorem}

\section{Preparation for the proofs}
\label{sec:3}

All the proofs use some modified form of the conjecture of Dickson \cite[1904]{Dic1904} about $k$-tuples of primes.
(His original conjecture referred for linear forms with integer coefficients.)
Let $\mathcal H = \{h_i\}_{i = 1}^k$ be a $k$-tuple of different non-negative integers.
We call $\mathcal H$ admissible, if
\beq
\label{eq:3.1}
P_{\mathcal H}(n) = \prod_{i = 1}^k (n - h_i)
\eeq
has no fixed prime divisor, that is, if the number $\nu_p(\mathcal H)$ of residue classes covered by $\mathcal H$ $\text{\rm mod }p$ satisfies
\beq
\label{eq:3.2}
\nu_p(\mathcal H) < p \ \text{ for } \ p \in \mathcal P.
\eeq
This is equivalent to the fact that the singular series
\beq
\label{eq:3.3}
\mathfrak S(\mathcal H) = \prod_p \left(1 - \frac{\nu_p(\mathcal H)}{p} \right) \left(1 - \frac1{p}\right)^{-k} > 0.
\eeq
Dickson conjectured that if $\mathcal H$ is admissible, then all $n + h_i$ will be primes simultaneously for infinitely many values of~$n$.

Hardy and Littlewood \cite[1923]{HL1923}, probably unaware of Dickson's conjecture, formulated a quantitative version of it, according to which
\beq
\label{eq:3.4}
\pi_{\mathcal H}(x) = \sum_{\substack{n \leq x\\
n + h_i \in \mathcal P(1 \leq i \leq k)}} 1 = \bigl(\mathfrak S(\mathcal H) + o(1) \bigr) \frac{x}{\log^k x}.
\eeq

In the work \cite[2009]{GPY2009} we attacked (but missed by a hair's breadth) the following weaker form of Dickson's conjecture which I called

\begin{ConjDHLk2}
If $\mathcal H$ is an admissible $k$-tuple, then $n + \mathcal H$ contains at least two primes for infinitely many values of~$n$.
\end{ConjDHLk2}

It is clear that if $\text{\rm DHL}(k,2)$ is proved for any $k$ (or even for any single $k$-tuple $\mathcal H_k$), then the Bounded Gap Conjecture is true.
$\text{\rm DHL}(k,2)$ was shown very recently by Y. T. Zhang \cite[2013]{Zhang2013} for $k \geq k_0 = 3.5 \times 10^6$ and this implied his Theorem~\ref{th:A}, the infinitude of gaps of size $\leq 7 \cdot 10^7$.

However, results of this type cannot exclude the existence of other primes and therefore give information on numbers expressible as difference of two primes, in the optimal case of Zhang's very strong Theorem~\ref{th:A} prove the existence of many weak Polignac numbers.
However, they do not provide more information about (strong) Polignac numbers than the very deep fact that $\mathcal D_s \neq \emptyset$ and they do not help in showing any of Theorems~\ref{th:1}--\ref{th:6}.
For example, in case of Theorem~\ref{th:6} they do not yield, let say, $4$-term arithmetic progressions of primes and a bounded number $d$ such that $p + d$ should also be prime for all four elements of the progression (even if we do not require that $p$ and $p + d$ should be \emph{consecutive} primes).

By a combination of the ideas of D. Goldston, C. Y{\i}ld{\i}r{\i}m, Y. Zhang and mine, I am able to show a much stronger form of Conjecture $\text{\rm DHL}(k,2)$ which may be applied towards the proof of Theorems \ref{th:1}--\ref{th:6}.
In case of Theorem~\ref{th:6} the ground-breaking ideas of Green and Tao \cite[2008]{GT2008} have to be used too, of course.

In view of Zhang's recent result the stronger form of Conjecture $\text{\rm DHL}(k,2)$ will obtain already the name Theorem $\text{\rm DHL}^*(k,2)$ for large enough $k$ and can be formulated as follows, first still as a conjecture for $k < 3.5 \cdot 10^6$.

Let $P^-(m)$ denote the smallest prime factor of~$m$.

\begin{ConjDHLstark2}
Let $k \geq 2$ and $\mathcal H = \{h_i\}_{i = 1}^k$ be any admissible $k$-tuple, $N \in Z^+$, $\varepsilon > 0$ sufficiently small $(\varepsilon < \varepsilon_0)$
\beq
\label{eq:3.4second}
\mathcal H \subset [0, H], \quad H \leq \varepsilon \log N, \quad P_{\mathcal H}(n) = \prod_{i = 1}^k (n + h_i).
\eeq
We have then positive constants $c_1(k)$ and $c_2(k)$ such that the number of integers $n \in [N, 2N)$ such that $n + \mathcal H$ contains at least two \emph{\rm consecutive} primes and almost primes in each components (i.e.\ $P^-(P_{\mathcal H}(n)) > n^{c_1(k)}$) is at least
\beq
\label{eq:3.5}
c_2(k)\mathfrak S(\mathcal H) \frac{N}{\log^k N} \ \text{ for $N > N_0(\mathcal H)$}.
\eeq
\end{ConjDHLstark2}

One can see that we have a looser condition than in $\text{\rm DHL}(k,2)$ as far as the elements of $\mathcal H$ are allowed to tend to infinity as fast as $\varepsilon \log N$.
On the other hand we get stronger consequences as

(i) we can prescribe that the two primes $n + h_i$ and $n + h_j$ in our $k$-tuple should be \emph{consecutive};

(ii) we have almost primes in each component $n + h_i$;

(iii) we get the lower estimate \eqref{eq:3.5} for the number of the required $n$'s with the above property.

The condition $n \in [N, 2N)$ makes usually no problem but in case of the existence of Siegel zeros some extra care is needed if we would like to have effective results (see Section 8 for this).

After this it is easy to formulate (but not to prove) our

\begin{maintheorem}
Conjecture $\text{\rm DHL}^*(k,2)$ is true for $k \geq 3.5 \times 10^6$.
\end{maintheorem}

\begin{proof}
Since this result contains Zhang's Theorem and even more, it is easy to guess that a self contained proof would be hopelessly long (and difficult).
We therefore try to describe only the changes compared to different earlier works.

The first pillar of Zhang's work is the method of proof of our Theorem~\ref{th:B}.
Although he supposes $\mathcal H$ as a constant the method of proof of Theorem~\ref{th:B} (see Propositions~1 and 2 in \cite[2009]{GPY2009})
allow beyond $H \ll \log N$ (required by \eqref{eq:3.4second} above) the much looser condition $H\ll N^{1/4 - \varepsilon}$ for any $\varepsilon > 0$.

The second pillar of Zhang's work is to show that distribution of primes according to non-smooth moduli, i.e.\ without any prime divisor $> N^b$ for any fixed small constant $b$, can be neglected.
As mentioned in the Introduction we showed this already much earlier in a joint work with Y. Motohashi \cite[2008]{MP2008}.
This work also supposed $\mathcal H$ to be a constant but the only place where actually more care is needed in the proof is (3.11) of \cite[2008]{MP2008}.
On the other hand, allowing here the condition
\beq
\label{eq:3.6}
H \ll \log N,
\eeq
the same simple argument as in Section 6 of \cite[2009]{GPY2009} adds an additional error term
\beq
\label{eq:3.7}
k^2 \log\log\log R \ll \log\log\log R
\eeq
to the right-hand side of (3.11) of \cite[2008]{MP2008} which is far less than the present error term $\log R_0 = \log R/(\log\log R)^5$.
Otherwise the proof works without any change, everything remains uniform under our condition \eqref{eq:3.6} above.

\begin{remark}
The crucial Lemmas 3 and 4 of our work \cite[2008]{MP2008} contain an additional factor $\gamma(n, \mathcal H)$.
However, by the definition (4.17) of \cite[2008]{MP2008} we have $\gamma(n, \mathcal H) = 1$ if $P^-\bigl(P_{\mathcal H}(n)\bigr) > R^\eta$ for any fixed $\eta > 0$.
In such a way the extra factor $\gamma(n, \mathcal H)$ does not affect the validity of our Lemma~\ref{lem:2} below since the asymptotic provided by Lemma~\ref{lem:3} for the right-hand side of \eqref{eq:3.10} (and similarly the analogue of it for primes, Lemma~\ref{lem:4}) is the same as if we used the constant weight $1$ instead of $\gamma(n, \mathcal H)$.
\end{remark}

No change is required in the third pillar of Zhang's work where he proves some sort of extension of the Bombieri--Vinogradov theorem for smooth moduli and the residue classes appearing by the method of Theorem~\ref{th:B}.

However, the proof of Theorem $\text{\rm DHL}^*(k,2)$ (for $k$ large enough, $k \geq 3.5 \cdot 10^6$) requires a further important idea, namely Lemmas 3--4 of the author's work \cite[2010]{Pin2010a}.
This we formulate now as

\begin{lemma}
\label{lem:1}
Let $N^{c_0} < R \leq \sqrt{N/p}(\log N)^{-C}$, $p \in \mathcal P$, $p < R^{C_0}$ with a sufficiently small positive $c_0$ and sufficiently large $C$.
Then we have with the notation
\beq
\label{eq:3.8}
\Lambda_R(n; \mathcal H, k + \ell) = \frac{1}{(k + \ell)!} \sum_{d \leq R, d \mid P_{\mathcal H}(n)} \mu(d) \left(\log \frac{R}{d}\right)^{k + \ell} \quad (\ell \leq k)
\eeq
the relation
\beq
\label{eq:3.9}
\sum_{\substack{n \in [N, 2N)\\
p\mid P_{\mathcal H}(n)}} \Lambda_R(n; \mathcal H, k + \ell)^2 \ll_k \frac{\log p}{p \log R} \sum_{n \in [N, 2N)} \Lambda_R(n; \mathcal H, k + \ell)^2.
\eeq
\end{lemma}

Lemma~\ref{lem:1} immediately implies

\begin{lemma}
\label{lem:2}
Let $N^{c_0} < R \leq N^{1/(2 + \eta)}(\log N)^{-C}, \quad \eta > 0$.
We have then
\beq
\label{eq:3.10}
\sum_{\substack{n \in [N, 2N)\\
P^- (P_{\mathcal H}(n)) < R^\eta}} \Lambda_R(n; \mathcal H, k + \ell)^2 \ll_k \eta \sum_{n \in [N, 2N)} \Lambda_R(n; \mathcal H, k + \ell)^2.
\eeq
\end{lemma}

\begin{remark}
Lemmas~\ref{lem:1} and \ref{lem:2} were already proved in \cite[2010]{Pin2010a} under the loose condition $H \ll \log N$.
\end{remark}

Lemma~\ref{lem:2} asserts that numbers $n$ where $P_{\mathcal H}(n)$ has a prime factor $< R^{\eta}$ (equivalently $<N^b$) with a small enough value of $\eta$ (or $b$) might be neglected, since the weight used in all proofs is actually of type \eqref{eq:3.8}.
The value of $\eta$ (or $b$) depends on $k$.

These results play a crucial role in the common generalization of the Green--Tao theorem and of our Theorem~\ref{th:B} (cf.\ \cite[2010]{Pin2010a}) and also in the proof that prime gaps $< \varepsilon \log p$ form a positive proportion of all gaps for any $\varepsilon > 0$ (proved in a joint work with Goldston and Y{\i}ld{\i}r{\i}m). 

These four pillars lead finally to the stronger form of Theorem $\text{\rm DHL}^*(k,2)$  if we combine it with a standard assertion following from Selberg's sieve, which we can formulate in this special case as

\begin{lemma}
\label{lem:3}
Let $0 < \alpha < 1/2$ be any constant.
Then
\beq
\label{eq:3.11}
\sum_{\substack{n \in [N, 2N)\\
P^-(P_{\mathcal H}(n)) > N^\alpha}} 1 \ll_k \frac{N\alpha^{-k}}{\log^k N} \mathfrak S(\mathcal H).
\eeq
\end{lemma}

\begin{proof}
This is Theorem 5.1 of \cite[1974]{HR1974} or Theorem~2 in \S 2.2.2 of \cite[2001]{Gre2001}.
This is also valid if we assume only $H \ll \log N$.
\end{proof}

We further need a generalization of Gallagher's theorem proved by the author which we formulate as

\begin{lemma}
\label{lem:4}
Let $\mathcal H_k$ be an arbitrary admissible $k$-tuple with
\beq
\mathcal H =  \mathcal H_k \subseteqq [0, H].
\label{3.11a}
\eeq
Then we have for any $\eta > 0$
\beq
S_{\mathcal H}(H) := \frac1{H} \sum_{h = 1}^{H} \frac{\mathfrak S(\mathcal H \cup h)}{\mathfrak S(\mathcal H)} = 1 + O(\eta)
\label{3.11b}
\eeq
if
\beq
H \geq \exp (k^{1/\eta}).
\label{3.11c}
\eeq
\end{lemma}

\begin{proof}
This is Theorem 1 of \cite[2010]{Pin2010}.
\end{proof}

\begin{remark}
As it is easy to see Lemma~\ref{lem:4} implies Gallagher's classical theorem \cite[1976]{Gal1976} on the singular series.
\end{remark}

Combining the proofs of Theorems~\ref{th:A} and \ref{th:B} (in the modified forms mentioned above) with the assertion of Lemma~\ref{lem:2} we obtain under the weaker condition \eqref{eq:3.5}, i.e.\ for all admissible $k$-tuples $\mathcal H_k = \{h_i\}_{i = 1}^k$, $h_i < h_{i + 1}$,
\beq
\label{eq:3.13}
\mathcal H = \mathcal H_k \subseteqq [0, \varepsilon \log N],
\eeq
at least
\beq
\label{eq:3.14}
c_2(k) \mathfrak S(\mathcal H) \frac{N}{\log^k N}
\eeq
numbers $n \in [N, 2N)$ such that $n + \mathcal H$ contains at least two primes and almost primes in each components, i.e.
\beq
\label{eq:3.15}
P^-\bigl(P_{\mathcal H}(n)\bigr) > n^{c_1(k)}.
\eeq
However, we must show the same with \emph{consecutive} primes as well.

We can define for any subset $V$ of $\{1,2, \dots, k\}$ the set
\beq
\label{eq:3.16}
V(N) = \bigl\{n \in [N, 2N) :\ n + h_i \in \mathcal P \Leftrightarrow i \in V\bigr\}.
\eeq
Since $k$ is bounded the number of possible subsets $V$ is also bounded, therefore we can choose a $V_0$ such that
\beq
\label{eq:3.17}
V_0 \subset \{1,2,\dots, k\}, \quad |V_0| \geq 2, \quad |V_0(N)| \geq c_3(k) \mathfrak S(\mathcal H) \frac{N}{\log^k N}.
\eeq

Choosing two arbitrary consecutive elements $i,j \in V_0(N)$ with $i < j$ we have at least
\beq
\label{eq:3.18}
c_3(k) \mathfrak S(\mathcal H) \frac{N}{\log^k N}
\eeq
numbers $n \in [N, 2N)$ such that for any  $\mu \in (i, j)$ 
\beq
\label{eq:3.19}
n + h_\mu \notin \mathcal P .
\eeq

We have to assure, however, additionally that for a positive proportion of these numbers $n$ we have also
\beq
\label{eq:3.20}
n + h \notin \mathcal P \ \text{ for } \ h \notin \mathcal H, \ \ h_i < h < h_j.
\eeq
Applying Lemma~\ref{lem:3} for all these values $h$ and summing up we arrive at the conclusion that
\begin{align}
\label{eq:3.21}
\sum_{\substack{n \in [N, 2N)\\ P^-(P_{\mathcal H}(n)) > N^{c_1(k)}}} \sum_{\substack{h\\ h_i < h < h_j\\ n + h \in \mathcal P}} 1
&\leq C_4(k) \frac{N}{\log^{k + 1} N} \sum_{h_i < h < h_j} \mathfrak S(\mathcal H \cup h)\\
&\leq 2C_4(k) \frac{N \varepsilon \mathfrak S(\mathcal H)}{\log^k N}
\notag
\end{align}
by Lemma~\ref{lem:4}.
This shows that if $\varepsilon $ was chosen sufficiently small depending on $k$, i.e.
\beq
\label{eq:3.22}
\varepsilon < \varepsilon_0(k)
\eeq
then our original primes $n + h_i$ and $n + h_j$ are \emph{consecutive} for at least
\beq
\label{eq:3.23}
c_5(k) \mathfrak S(\mathcal H) \frac{N}{\log^k N}
\eeq
elements $n \in [N, 2N)$, thereby showing our Main Theorem.
\end{proof}

\section{Polignac numbers}
\label{sec:4}

1. Proof of Theorem~\ref{th:1}

The fact that $\text{\rm DHL}^*(k, 2)$ implies the positivity of the asymptotic lower density of Polignac numbers is expressed as Corollary~1 of \cite[2010]{Pin2010a} with the value
\beq
\label{eq:4.1}
\frac{1}{k(k - 1)} \prod_{p \leq k} \left(1 - \frac1p\right)
\eeq
and proved in Section~11.
The above value is about $e^{-\gamma} / (k^2 \log k)$ for large values of $k$.
(If $k \to \infty$, they are asymptotically equal.)

1. Proof of Theorem~\ref{th:2}

We again suppose $\text{\rm DHL}^*(k, 2)$ for $k \geq k_0$.

The reasoning is more intricate in this case and the resulting value $C$ is ineffective.
The trivial relation
\beq
\label{eq:4.2}
h_k - h_1 \geq (k - 1) \min_{i > j} (h_i - h_j) \quad (h_i < h_{i + 1})
\eeq
even gives the impression that the best we can prove about localization of Polignac numbers is an interval of type
\beq
\label{eq:4.3}
\bigl[M, (k - 1) M\bigr].
\eeq

Let us suppose that Theorem~\ref{th:2} is false.
Then we have for any $C_0 > 0$ an infinite series of intervals
\beq
\label{eq:4.4}
I_\nu := \bigl[M_\nu, M_\nu + C_\nu\bigr], \quad M_\nu > C_\nu > 4M_{\nu - 1}, \quad M_1 > C_0
\eeq
such that
\beq
\label{eq:4.5}
\mathcal D_s \cap \biggl(\bigcup\limits_{\nu = 1}^\infty I_\nu \biggr) = \emptyset.
\eeq

For $p > k$ we have clearly
\beq
\label{eq:4.6}
\nu_p (\mathcal H_k) < p \ \text{ for } \ p \in \mathcal P,
\eeq
so we have no problem of choosing an admissible system $\mathcal H_k$ in a sufficiently long interval (e.g.\ if $C_k$ is large enough).
Let
\beq
\label{eq:4.7}
\mathcal H_k := \{h_\nu\}_{\nu = 1}^k\,, \quad h_\nu \in I_\nu':= \bigl[M_\nu + C_\nu / 2, M_\nu + C_\nu \bigr].
\eeq
For $h_\nu \in I_\nu$, $h_\mu \in I_\mu$, $\nu < \mu$ we have then
\beq
\label{eq:4.8}
h_\mu - h_\nu \in \bigl[ M_\mu + C_\mu / 2 - 2M_{\mu - 1}, M_\mu + C_\mu \bigr] \subset I_\mu.
\eeq

Since in case of $k \geq k_0$ at least one of the numbers
\beq
\label{eq:4.9}
h_\mu - h_\nu \qquad (1 \leq \nu < \mu \leq k)
\eeq
can be written as a difference of two consecutive primes, \eqref{eq:4.8} contradicts to \eqref{eq:4.5} and thus proves our theorem.

\section{The normalized value distribution of $d_n$}
\label{sec:5}

1. Proof of Theorems \ref{th:3}--\ref{th:4}

Since Theorem \ref{th:4} implies Theorem \ref{th:3} it is sufficient to prove the latter.
The structure of the proof will follow that of Theorem \ref{th:2} proved in the previous section.
Suppose that Theorem \ref{th:4} is false.
In this case we have for a sufficiently small $c^* > 0$ an infinite series of intervals
\beq
\label{eq:5.1}
J_\nu = \bigl[c_\nu, c_\nu + \delta_\nu\bigr], \quad c_\nu > 4\delta_\nu > 20 c_{\nu + 1}, \quad c_1 < c^*
\eeq
such that for $K$ large enough
\beq
\label{eq:5.2}
\left\{ \frac{d_n}{f(n)} \right\}_{n = N}^\infty \cap \biggl(\bigcup_{\nu = 1}^K J_\nu \biggr) = \emptyset, \ \text{ where }
N = N(K) > 0.
\eeq

Let
\beq
\label{eq:5.3}
I_\nu(n) := \bigl[c_\nu f(n), (c_\nu + \delta_\nu) f(n)\bigr], \quad \nu = 1,2, \dots, K.
\eeq
Then we have
\beq
\label{eq:5.4}
d_n \notin \bigcup_{\nu = 1}^K I_\nu(n) \ \text{ for }\ \nu = 1,2, \dots, K, \ \ n \in [N, 2N), \ \ N > N(k).
\eeq
Using our Main Theorem, similarly to \eqref{eq:4.7} we can construct an admissible $k$-tuple
$\mathcal H_k = \{h_\nu\}_{\nu = 1}^k$ with $3.5 \cdot 10^6 \leq k \leq K$ such that
$h_1 > h_2 > \dots > h_k$ and with a sufficiently small $\varepsilon > 0$
\beq
\label{eq:5.5}
h_\nu \in I_\nu'(N) := \left[\left(c_\nu + \frac{\delta_\nu}{2}\right) (1 + \varepsilon) f(N), (c_\nu + \delta_\nu)(1 - \varepsilon) f(N) \right].
\eeq
For $h_\mu \in I_\mu'$, $h_\nu \in I_\nu'$, $\mu < \nu$ we have for $N > \max \bigl(N(K), N_0(\varepsilon)\bigr)$
\beq
\label{eq:5.6}
h_\mu - h_\nu \in \left[\left(c_\mu + \frac{\delta_\mu}{2} - 2c_{\mu + 1}\right) (1 + \varepsilon) f(N), (c_\mu + \delta_\mu)(1 - \varepsilon) f(N)\right] := I_\mu^*(N).
\eeq

Since we have for any $1 \leq \mu < \nu \leq k_0$ and any $n \in [N, 2N)$
\beq
\label{eq:5.7}
I_\mu^*(N) \subset I_\mu(n),
\eeq
the fact that by the Main Theorem $h_\mu - h_\nu = d_n$ for some $n \in [N, 2N)$, $1 \leq \mu < \nu \leq k_0$ contradicts to \eqref{eq:5.4} and thus proves Theorems \ref{th:3} and \ref{th:4}.

\section{Comparison of two consecutive values of $d_n$}
\label{sec:6}

1. Proof of Theorem \ref{th:5}

Since the proof of the two inequalities are completely analogous, we will only prove the second one.
The basis of it is the Main Theorem proved in Section~3.
We can start with an arbitrary admissible $k$-tuple $\mathcal H = \mathcal H_k = \{h_i\}_{i = 1}^k$, $h_1 < h_2 < \dots < h_k$ with $k \geq 3.5\cdot 10^6$.
Let with a fixed sufficiently small $c_1(k)$ define
\beq
\label{eq:6.1}
B(i,j,N) = \left\{n \leq N, n + h_i \in \mathcal P, n + h_j \in \mathcal P, {\mathcal P}^-(\mathcal P_{\mathcal H}(n)) > n^{c_1(k)}\right\},
\eeq
\beq
\label{eq:6.2}
\mathcal T = \left\{ (i,j); j > i, \limsup_{N \to \infty} \frac{|B(i, j, N)|\log^k N}{N} > 0\right\}
\eeq
and let us choose the pair $(i,j)$ with maximal value of $j$, afterwards that with maximal value of $i < j$.
Then for any $h_\mu \in (h_i, h_j)$ (i.e.\ $i < \mu < j$) we have clearly
\beq
\label{eq:6.3}
\limsup_{N \to \infty} \frac{|B(\mu, j, N)| \log^k N}{N} = 0
\eeq
so all components $n + h_\mu$ between $n + h_i$ and $n + h_j$ are almost always composite if $n \in B(i, j, N)$ and $N = N_\nu \to \infty$ through a suitable sequence $N_\nu$.

On the other hand if we have an arbitrary $h \in (h_i, h_j)$, $h \notin \mathcal H$, then the assumption $n + h \in \mathcal P$ implies for $\mathcal H^+ = \mathcal H \cup h$
\beq
\label{eq:6.4}
P^-\bigl(P_{\mathcal H^+}(n)\bigr) > n^{c_1(k)}.
\eeq
However, by Lemma~\ref{lem:3} the number of such $n \leq N$ is for all $N$
\beq
\label{eq:6.5}
\ll_{k, c_1} \frac{\mathfrak S(\mathcal H \cup \{h\})N}{\log^{k + 1}N} \ll_{k, c_1} \frac{\mathfrak S(\mathcal H)N \log h_k}{\log^{k + 1} N} \ll_{k, c_1, \mathcal H} \frac{N}{\log^{k + 1} N}.
\eeq
This, together with \eqref{eq:6.3} shows that we have at least
\beq
\label{eq:6.6}
\bigl(c_1(k, \mathcal H) + o(1)\bigr) \frac{N}{\log^k N}
\eeq
values $n \leq N$ with $n + h_i$, $n + h_j$ being consecutive primes for some sequence $N = N_\nu \to \infty$.

Let us consider now these differences.
Let
\beq
\label{eq:6.7}
n + h_i = p_\nu \in \mathcal P, \quad n + h_j = p_{\nu + 1} \in \mathcal P, \quad
d_\nu = h_j - h_i \ll 1
\eeq
where
\beq
\label{eq:6.8}
\log n \sim \log\nu \sim \log N.
\eeq

Suppose now that the second inequality of Theorem \ref{th:5} is false.
Then we have for all those values of $\nu$ with an arbitrary $\varepsilon > 0$ for $N > N(\varepsilon)$
\beq
\label{eq:6.9}
d_{\nu + 1} \leq d_\nu \varepsilon \log N \leq C \varepsilon \log N, \ \ C = h_k - h_1 .
\eeq
The already quoted sieve of Selberg (Lemma \ref{lem:3}) gives an upper estimate how often this might happen for any particular value
\beq
\label{eq:6.10}
d_{\nu + 1} = d \leq C \varepsilon \log N.
\eeq
Adding it up until $C \varepsilon \log N$ we obtain at most
\beq
\label{eq:6.11}
\ll_{k, c_1} \frac{N}{\log^{k + 1} N} \sum_{h = 1}^{C \varepsilon \log N} \mathfrak S(\mathcal H \cup h) \ll_{k, c_1} \mathfrak S(\mathcal H) \frac{CN}{\log^k N} \varepsilon
\eeq
by Lemma \ref{lem:4}.
This means that in view of \eqref{eq:6.6} this cannot hold for all $N = N_\nu \to \infty$.
This contradiction proves Theorem \ref{th:5}.

\section{Arbitrarily long arithmetic progressions of generalized twin primes}
\label{sec:7}

Proof of Theorem \ref{th:6}

In this case we have to use again our crucial Main Theorem and the rest of the machinery executed in Section~7 of \cite[2010]{Pin2010a}.
This yields the combination of Zhang's theorem with that of Green and Tao \cite[2008]{GT2008}.

\section{How to make Zhang's theorem effective?}
\label{sec:8}

Our last point is that in its original form Zhang's theorem is ineffective since it uses Siegel's theorem.
His result and similarly to it all results of the present work can be made effective in the following way.

According to the famous theorem of Landau--Page there is at most one real primitive character $\chi$ with a modulus $q$ (and the characters induced by it) which might cause ineffectivity in the Bombieri--Vinogradov theorem, and this modulus has to satisfy
\beq
\label{eq:8.1}
(\log X)^2 \leq q \leq (\log X)^{\omega(X)}
\eeq
for any $\omega(X) \to \infty$ as $X \to \infty$.
This modulus can cause any problem only in the case
\beq
\label{eq:8.2}
q \mid P_{\mathcal H}(n).
\eeq

However, our Lemma~\ref{lem:3} says that we can neglect all numbers $n$ with
\beq
\label{eq:8.3}
P^-\bigl(P_{\mathcal H}(n)\bigr) < n^{c_1(k)},
\eeq
so (prescribing additionally $q \nmid d$ in the definition \eqref{eq:3.8} of $\Lambda_R(n; \mathcal H, k, \ell)$) both Zhang's theorem and all our present results become effective.

\begin{remark}
It is an interesting phenomenon that in the sieving process yielding bounded gaps between primes and in all our present results we can choose
\beq
\label{eq:8.4}
\lambda_d = 0
\eeq
if $d$ has either a prime divisor
\beq
\label{eq:8.5}
> N^{c'(k)}
\eeq
or if it has a prime divisor
\beq
\label{eq:8.6}
< N^{c''(k)}.
\eeq
\end{remark}

\end{document}